\documentclass[12pt]{article}
\usepackage{amsmath,amssymb,amsbsy,amsfonts,amsthm,latexsym,
                        amsopn,amstext,amsxtra,euscript,amscd}
\begin{document}
\newtheorem{theorem}{Theorem}
\newtheorem{lemma}[theorem]{Lemma}
\newtheorem{claim}[theorem]{Claim}
\newtheorem{cor}[theorem]{Corollary}
\newtheorem{prop}[theorem]{Proposition}
\newtheorem{definition}{Definition}
\newtheorem{question}[theorem]{Open Question}

\def\cA{{\mathcal A}}
\def\cB{{\mathcal B}}
\def\cC{{\mathcal C}}
\def\cD{{\mathcal D}}
\def\cE{{\mathcal E}}
\def\cF{{\mathcal F}}
\def\cG{{\mathcal G}}
\def\cH{{\mathcal H}}
\def\cI{{\mathcal I}}
\def\cJ{{\mathcal J}}
\def\cK{{\mathcal K}}
\def\cL{{\mathcal L}}
\def\cM{{\mathcal M}}
\def\cN{{\mathcal N}}
\def\cO{{\mathcal O}}
\def\cP{{\mathcal P}}
\def\cQ{{\mathcal Q}}
\def\cR{{\mathcal R}}
\def\cS{{\mathcal S}}
\def\cT{{\mathcal T}}
\def\cU{{\mathcal U}}
\def\cV{{\mathcal V}}
\def\cW{{\mathcal W}}
\def\cX{{\mathcal X}}
\def\cY{{\mathcal Y}}
\def\cZ{{\mathcal Z}}

\def\A{{\mathbb A}}
\def\B{{\mathbb B}}
\def\C{{\mathbb C}}
\def\D{{\mathbb D}}
\def\E{{\mathbb E}}
\def\F{{\mathbb F}}
\def\G{{\mathbb G}}
\def\I{{\mathbb I}}
\def\J{{\mathbb J}}
\def\K{{\mathbb K}}
\def\L{{\mathbb L}}
\def\M{{\mathbb M}}
\def\N{{\mathbb N}}
\def\O{{\mathbb O}}
\def\P{{\mathbb P}}
\def\Q{{\mathbb Q}}
\def\R{{\mathbb R}}
\def\S{{\mathbb S}}
\def\T{{\mathbb T}}
\def\U{{\mathbb U}}
\def\V{{\mathbb V}}
\def\W{{\mathbb W}}
\def\X{{\mathbb X}}
\def\Y{{\mathbb Y}}
\def\Z{{\mathbb Z}}

\def\E{{\mathbf E}}
\def\Fp{\F_p}
\def\ep{{\mathbf{e}}_p}
\def\Nm{{\mathrm{Nm}}}
\def\lcm{{\mathrm{lcm}}}

\def\scr{\scriptstyle}
\def\\{\cr}
\def\({\left(}
\def\){\right)}
\def\[{\left[}
\def\]{\right]}
\def\<{\langle}
\def\>{\rangle}
\def\fl#1{\left\lfloor#1\right\rfloor}
\def\rf#1{\left\lceil#1\right\rceil}
\def\le{\leqslant}
\def\ge{\geqslant}
\def\eps{\varepsilon}
\def\mand{\qquad\mbox{and}\qquad}

\newcommand{\comm}[1]{\marginpar{%
\vskip-\baselineskip 
\raggedright\footnotesize
\itshape\hrule\smallskip#1\par\smallskip\hrule}}

\def\xxx{\vskip5pt\hrule\vskip5pt}

\def\rank{\mathop{rank}}
\def\trace{\mathop{trace}}
\def\ind{\mathrm{ind}}
\def\IM{\mathrm{Im}}


\title{{\bf Short Cycles in Repeated Exponentiation 
Modulo a Prime}}

\author{
         {\sc{Lev Glebsky}} \\
         {Instituto de Investigaci{\'o}n en Comunicacin {\'O}ptica}\\   
          {Universidad Aut{\'o}noma de San Luis Potos{\'i}} \\
         {Av. Karakorum 1470, Lomas 4a 78210}\\
          {San Luis Potosi, Mexico} \\
         {\tt glebsky@cactus.iico.uaslp.mx}\\
        \and
         {\sc{Igor E.~Shparlinski}} \\
         {Department of Computing,  Macquarie University} \\
         {Sydney, NSW 2109, Australia} \\
         {\tt igor@ics.mq.edu.au}}
\date{\today}
\pagenumbering{arabic}

\maketitle

\begin{abstract}
Given a prime $p$, we consider the dynamical system generated by 
repeated exponentiations modulo $p$, that 
is,  by the map  $u \mapsto f_g(u)$, where 
$f_g(u) \equiv g^u \pmod p$ and $0 \le f_g(u) \le p-1$. 
This map is in particular 
used in a number of constructions of cryptographically secure 
pseudorandom generators.  We obtain nontrivial upper bounds
on the number of fixed points and short cycles in the above
dynamical system. 
\end{abstract}

\section{Introduction}  

Given a prime $p$ and an integer $g$ with $\gcd(g,p)=1$ 
one can consider the dynamical system generated by consecutive 
exponentiations  modulo $p$ where $g$ serves as the base.
More precisely, we define the function $f_g(u)$ by the conditions
$$
f_g(u) \equiv g^u \pmod p \mand 0 \le f_g(u) \le p-1,
$$
and for some initial value $u_0$ we consider sequences of 
consecutive iteration 
\begin{equation}
\label{eq:Seq}
u_n = f_g(u_{n-1}), \qquad n =1,2, \ldots.
\end{equation}

Besides of being of  intrinsic interest,  
this map has been used in several construction, 
see~\cite{GoldRos,PaSu} and references therein.

Here we study  
the number of initial values $u_0 \in \{1, \ldots, p-1\}$
which lead to short cycles. 
More precisely, for an integer $k$ we denote by 
$N_g(k)$ the number of $u_0 \in \{1, \ldots, p-1\}$
such that for the sequence~\eqref{eq:Seq} we
have $u_k = u_0$.

The quantities $N_g(1)$ (that is the number of 
fixed points) and  $N_g(2)$ have recently been
studied in~\cite{CobZah,Hold,HoldMor1,HoldMor2,Zhang}
``on average'' over $g \in  \{1, \ldots, p-1\}$. 
However, here we are mostly interested in ``individual''
results when $g$ is fixed. 

We remark that questions of this kind   can  be 
reformulated in an equivalent form as questions about 
iterations of the discrete logarithm function.

It is also important to note that generally speaking
$$
f_g(f_g(u)) \not \equiv g^{g^u} \pmod p. 
$$
In particular, the results of~\cite{MeWi} can be used to estimate 
the number of solutions to 
$$
g^{g^{u_0}}  \equiv u_0 \pmod p, \qquad u_0 \in \{1, \ldots, p-1\},
$$
but  do not apply to $N_g(2)$ directly.

\section{Preparation}  
  
We repeatedly use the following simple statement:

\begin{lemma}
\label{lem:fact1}
Let $u \equiv v\mod p$ and $v\in \{0,1,\ldots,p-1\}$, then 
$$
g^v\equiv g^{u-\fl{u/p}} \mod p.
$$
\end{lemma} 

\begin{proof}  Write 
$$
v= u - pw,
$$
where 
$$
w=\fl{\frac{u}{p}}  . 
$$
Then 
$$
g^v \equiv g^{u - pw} \equiv  g^{u - w} \pmod p.
$$
since $g^{p-1} \equiv 1 \pmod p$.
\end{proof}

\section{Fixed Points}

In the case when $g$ is a primitive root one can easily derive from 
a much more general result of~\cite[Theorem~1]{CoSh} that 
$N_g(1) = O(p^{1/2})$. 

Here we give a self-contained proof which also applies to any 
$g$. 

\begin{theorem}\label{th_N1}
For $p \ge 11$, 
uniformly over all integer $g$ with $\gcd(g,p)=1$ 
we have
$N_g(1)\leq \sqrt{2p}+1/2$.
\end{theorem}

\begin{proof} Let $1 \le x_1 < \ldots < x_N \le p-1$,
where $N =N_g(1)$, satisfy
$$
f_g(x_i) = x_i , \qquad i =1, \ldots, N.
$$
We consider the differences 
$x_i - x_j$, $1 \le j<i \le N$. 
Since $1\le x_i - x_j \le p-2$, 
at least one difference, say  $a$,  is taken 
at least 
$$
T \ge \frac{N(N-1)}{2(p-2)}
$$ 
times. 
Thus 
$$
(x_j + a) \equiv g^ax_j \pmod p
$$ 
for at least $T$ values of $j =1, \ldots, N$.
This immediately implies that $T \le 1$ and the result 
follows. \end{proof}

\section{Cycles of Length Two}

Unfortunately, in the case of cycles of length two and three
our method works only for small values of $g$. 

\begin{theorem}\label{th_N2}
For any fixed   integer $g$ with $\gcd(g,p)=1$ 
we have
$$
N_g(2)\le C(g) \frac{p}{\log p}
$$
where $C(g)$ depends only on $g$. 
\end{theorem}

\begin{proof} Let $1 \le y_1 , \ldots y_N \le p-1$,
where $N =N_g(2)$, satisfy
$$
 f_g(x_i)  = y_i \mand f_g(y_i)  = x_i, \qquad i =1, \ldots, N.
$$
for some  $x_i=y_{j_i}$.

Let us choose a real positive parameter $z<p$ to be optimized later. 
Clearly, there are at most 
\begin{equation}
\label{eq:I0}
I_{0} \le p/z+1 
\end{equation}
values of $i =1, \ldots, N$ with 
$y_{i+1} - y_i \ge z$ or $i = N$.

Now, for every positive $a < z$ we count the number $I_a$ of
$i =1, \ldots, N-1$  with $y_{i+1} - y_i = a$.
For such $i$, from 
$$
 y_i \equiv g^{x_i}   \pmod p \mand  
 y_{i+1}   \equiv g^{x_{i+1}}   \pmod p 
$$
 we derive
\begin{equation}
\label{eq:Eq1}
g^{x_{i+1}} \equiv  g^{x_i} + a \pmod p. 
\end{equation}

Furthermore, since both $x_i$ and $x_{i+1}$ generate cycles
of length two, we have 
$$
g^{y_{i}} \equiv  x_{i} \pmod p
\mand g^{y_{i+1}} \equiv  x_{i+1} \pmod p
$$
which yields
$$
x_i g^a \equiv g^{y_{i}+a} \equiv g^{y_{i+1}} \equiv x_{i+1} \pmod p.
$$
Thus, by Lemma~\ref{lem:fact1}
\begin{equation}
\label{eq:Eq2}
g^{x_{i+1}} \equiv g^{x_i g^a - k} \pmod p, 
\end{equation}
where 
$$
k=\fl{\frac{x_ig^a}{p}} <  g^a.
$$
Combining~\eqref{eq:Eq1} and~\eqref{eq:Eq2} we
obtain
$$
y_i + a  \equiv  g^{-k}y_i^{g^a} \pmod p.
$$
Clearly, for every $a$ and $k$, this is  a nontrivial polynomial 
equation of 
degree at most $g^a$ (to see this it is enough to compare 
the polynomials $Y+a$ and $g^kY^{g^a}$ at $Y = 0$). 
Hence for every $k$ there are at most $g^a$ possible
values of $y_i$.
Because $k$ takes at most  $g^a$ possible values, we obtain 
\begin{equation}
\label{eq:Ia}
I_a \le g^{2a}, \qquad a < z. 
\end{equation}
Therefore, we see from~\eqref{eq:I0} and~\eqref{eq:Ia} that
$$
N_g(2)\le 2I_0 + 2\sum_{a<z} I_a \le  2p/z +2 + 2\frac{g^{2z}-1}{g-1}\leq
2p/z+2+2g^{2z}.
$$
Taking
$$
z = \rf{\frac{\log p}{3 \log g}}
$$
we conclude the proof. \end{proof}

\section{Cycles of Length Three}

Here we use $\Z_n=\{0,1,\ldots,n-1\}$ to denote the 
residue ring modulo $n$. We also use $\F_p = \Z_p$ 
to denote the finite field of $p$ elements. 

It is convenient to  denote by
$a\oplus_n b$  the sum of integers $a$ and $b$ modulo $n$;
 so, $a\oplus_n b \equiv a+b \pmod n$ and $0\geq a\oplus_n b\geq n-1$.  
  
We start with the following simple statement:
 
\begin{lemma}
\label{lem:fact2}
Let $g,y\in \{1, \ldots, p-1\}$. If 
$$
\fl{\frac{gy+g}{p}}> \fl{\frac{gy+1}{p}}
$$ 
then
$$
y\in \left\{\fl{\frac{p}{g}},\fl{\frac{2p}{g}},\ldots,
\fl{\frac{(g-1)p}{g}}, p-1\right\}.
$$
\end{lemma}

\begin{proof}
We have 
$$\fl{\frac{gy+g}{p}}= \fl{\frac{gy+1}{p}}+ 1=k+1
$$
for some positive integer 
$$
k \le \frac{g(p-1)+1}{p} < g  .
$$
In particular 
$$
\frac{gy+1}{p} < k+1
$$
thus 
$$
y< \frac{gy + 1}{g}<\frac{p(k+1)}{g}\leq y+1
$$
which gives the desired result.
\end{proof}

We also need   the following combinatorial
result, that could be of independent interest: 

\begin{lemma}\label{lemma_comb1}
Consider two arbitrary  sets $\cM, \cS\subseteq \Z_n$ and
also define $\cC=\left\{x\in \cM\ \mid \ x\oplus_n1\in \cM\right\}$. 
Suppose, there exists a map
$\varphi:\cC\setminus \cS\to \Z_n\setminus \cM$, such that 
the cardinality of 
the preimage $\varphi^{-1}(a)$ of $a$ satisfies 
$\# \varphi^{-1}(a)\leq k$ for any
 $a\in \Z_n$. Then 
$$
 \# \cM\leq \frac{k+1}{k+2}n +\frac{\# \cS}{k+2}.
 $$ 
\end{lemma}

\begin{proof}
We split  
$\cM=\cI_1\cup\cdots \cup \cI_r$ into $r$ intervals 
of the form 
$$
\cI_j = \left\{x_j,x_j\oplus_n 1\ldots x_j\oplus_n (h_j-1)\right\}\subseteq \cM, 
\mand x_j\oplus_n h_j \not \in \cM, .
$$
where $j=1, \ldots, r$. Thus  between $I_j$ and 
$I_{j\oplus_r 1}$ there is an element in $\Z_n\setminus M$. 
This implies the inequalities
\begin{equation}
\label{eq:ineq1}
\# \cC \ge \# \cM - r
\end{equation}
and 
\begin{equation}
\label{eq:ineq2}
n\geq  \# \cM +r  .
\end{equation}

On the other hand, since $\varphi^{-1}(\Z_n\setminus \cM)=\cC\setminus \cS$ we derive 
$$
\#\(\Z_n\setminus \cM\) \geq \frac{\#\cC-\#\cS}{k}.
$$ 
So, recalling~\eqref{eq:ineq1} and~\eqref{eq:ineq2}, one has
$$
n  \geq  \# \cM +  \frac{\#\cC-\#\cS}{k} 
\ge \# \cM + \frac{\# \cM-r-\# \cS}{k}
\ge   \# \cM + \frac{2 \# \cM-n-\# \cS}{k}
$$
that  implies the desired inequality. 
\end{proof}

\begin{theorem}\label{th_N(3)}
Let $p$ be a prime. Then
$N_g(3)\leq \frac{3}{4}p+\frac{g^{2g+1}+g+1}{4}$.
\end{theorem}

\begin{proof}
All periodic points of $f_g$ belongs to the multiplicative subgroup generated
by $g$. So, the bound holds if $g$ is not a primitive root modulo $p$.
In what follows we assume that $g$ is a primitive root modulo $p$
and as usual, for an integer $u$ with $\gcd(u,p)=1$ we define $\ind_g u$
as the unique nonnegative integer $v\le p-2$ with $g^v \equiv u \pmod p$.

Let $\cN_3$ be the set of $u \in \{1, \ldots, p-1\}$ which 
generate a cycle of length  three.

We consider  
\begin{equation}
\label{eq:good x}
x_1\in\cN_3 \setminus \left\{p-1,\ind_g(p-1),\ind_g\(\fl{p/g}\),\ldots,
\ind_g(\fl{(g-1)p/g}) \right\}
\end{equation}
and put
$$
y_1 = f_g(x_1),\qquad  z_1= f_g(x_1), \qquad x_1=f_g(z_1).
$$
Suppose also that for $x_2 = x_1 \oplus_p 1$  we also have 
$x_2 \in \cN_3$, that is,
$$
y_2 = f_g(x_2),\qquad  z_2= f_g(y_2), \qquad x_2=f_g(z_2).
$$ 

We have 
$$
y_2 \equiv y_1 g \pmod p.
$$
By  Lemma~\ref{lem:fact1}, we have
$$
z_2 \equiv g^{y_2} \equiv g^{g y_1 - \fl{gy_1/p}}  \equiv z_1^{g} g^{- \fl{gy_1/p}}  \pmod p.
$$
Then $z_1$ satisfies the following congruence:
\begin{equation}\label{eq_1}
x_1 + 1 \equiv x_2 \equiv  g^{z_2}\equiv g^{u z_1^g-\fl{u z_1^g/p}} \pmod p,
\end{equation}
where $u\equiv g^{-\fl{gy_1/p}} \pmod p$, $0 \le u \le p-1$.

Finally, suppose, that $x_3=\ind_g(y_1+1)$ and $x_4= x_3 \oplus_p 1$ both 
satisfy $x_3, x_4 \in \cN_3$. We put
$$
y_3 = f_g(x_3),\qquad  z_3= f_g(y_3), \qquad x_3=f_g(z_3)
$$ 
and 
$$
y_4 = f_g(x_4),\qquad  z_4= f_g(y_4), \qquad x_4=f_g(z_4).
$$
In particular,
\begin{equation}
\label{eq:Mult 1}
z_3\equiv gz_1 \pmod p
\end{equation}
and
\begin{equation}
\label{eq:Mult 2}
 y_4  \equiv gy_3 \equiv gy_1 + g \pmod p.
\end{equation}
Using~\eqref{eq:Mult 1} and Lemma~\ref{lem:fact1}, 
we have
\begin{equation}\label{eq_2}
x_3 \equiv  g^{z_3}  \equiv   g^{gz_1-\fl{gz_1/p}} \pmod p,
\end{equation}

Similarly, by~\eqref{eq:Mult 2}
and Lemma~\ref{lem:fact1}, we have
$$
z_4 \equiv g^{y_4}   \equiv v g^{gy_1 + g} \pmod p
$$
with some integer 
$$
v\equiv g^{-\fl{g(y_1 + 1)/p}} \pmod p \mand 0 \le v \le p-1.
$$
Thus
$$
z_4 \equiv  g^g v z_1^g \pmod p
$$
and by Lemma~\ref{lem:fact1}, we derive
\begin{equation}\label{eq_3}
x_4 \equiv g^{z_4} \equiv g^{g^g v z_1^g-\fl{g^g v z_1^g/p}} \pmod p
\end{equation}

Since $x_4 \equiv x_3 + 1$, we derive from~\eqref{eq_2} 
and~\eqref{eq_3} that 
$$
g^{gz_1-\fl{gz_1/p}} +1 \equiv 
g^{g^g v z_1^g-\fl{g^g v z_1^g/p}} \pmod p.
$$
Since the condition~\eqref{eq:good x}, we see that $u=v$.
Therefore
\begin{eqnarray*}
g^{gz_1-\fl{gz_1/p}} +1 & \equiv & g^{g^g u z_1^g-\fl{g^g u z_1^g/p}}\\
 & \equiv & \(g^{u z_1^g - \fl{u z_1^g/p}}\)^{g^g}
g^{g^g\fl{u z_1^g/p}-\fl{g^g u z_1^g/p}} \pmod p.
\end{eqnarray*}

Recalling~\eqref{eq_2} and using $x_1 \equiv g^{z_1} \pmod$, 
we see that the last congruence is equivalent to
$$
x_1^g g^{-\fl{gz_1/p}} +1   \equiv 
 \(x_1 + 1\)^{g^g}
g^{g^g\fl{u z_1^g/p}-\fl{g^g u z_1^g/p}} \pmod p. 
$$
 
Now, 
$$
\fl{\frac{gz_1}{p}}\in\left\{0,1,\ldots,g-1\right\}
$$ 
and 
$$
\fl{\frac{g^guz_1^g}{p}}-g^g\fl{\frac{uz_1^g}{p}}\in
\left\{0,\cdots g^g-1\right\}.
$$
So, $x_1$ satisfies one of $g^{g+1}$ possible nontrivial 
polynomial congruences of degree $g^g$.
Let $\cX$ denote the set of all such $x_1$. 
We now consider the set 
$$
\cS=\cX\bigcup
\left\{p-1,\ind_g(p-1),\ind_g\fl{\frac{p}{g}},
\ind_g\fl{\frac{2p}{g}},\ldots,
\ind_g\fl{\frac{(g-1)p}{g}} \right\} 
$$
of cardinality 
$$
\# \cS \leq g^{2g+1}+g+1. 
$$

Then for every $x \not\in S$ if
$$
x,x\oplus_p 1,f_g(x)\oplus_p 1\in \cN_3
$$
then 
$$
f_g(f_g(f_g(x)\oplus_p 1))\oplus_p 1\not\in \cN_3.
$$
We  put $\cM =\cN_3$ and define the set $\cC$ as in Lemma~\ref{lemma_comb1}. 
We now construct function 
$\varphi:\qquad \cC\setminus \cS \to \F_p \setminus \cM$
by the following rule:
$$
\varphi(x)=\left\{\begin{array}{ccc} 
f_g(x)\oplus_p 1               & \mbox{if} & f_g(x)\oplus_p 1\not\in \cM,\\
f_g(f_g(f_g(x)\oplus_p 1))\oplus_p 1 & \mbox{if} & f_g(x)\oplus_p 1\in \cM.
\end{array}\right.
$$   
Since both functions used in the definition of $\varphi$ 
are invertible  functions, we have $|\varphi^{-1}(a)|\leq 2$ 
and thus we can apply  Lemma~\ref{lemma_comb1} with  $k=2$,
which concludes the proof. 
\end{proof}

\section{Open Questions}

We have no doubts that our estimates are very far from 
the true behaviour of $N_g(1)$, $N_g(2)$ and $N_g(3)$.
Yes, they seem to be the only known results. 
Unfortunately, our approach does not work for $N_g(k)$ 
with $k \ge 4$ and finding an alternative way to estimate,
say $N_g(4)$ is an important open question. 

One can also consider analogues of our results for elliptic
curves. Namely, let $\cE$ be an elliptic curve over  $\F_p$ 
given by an affine Weierstra\ss\ equation:
$$
\cE~:~Y^2 = X^3 + aX + b.
$$
It is well-known that $\cE$ has a structure of 
a  finite  abelian group  under an appropriate
composition rule, with the point at infinity $\cO$ as the neutral
element, see~\cite{Silv}.

Furthermore, given an $\F_p$-rational point $P \in \cE$
we denote by $x(P)$ its $x$-coordinate. 
Using the groups structure of points on $\cE$, 
for a point $G\in \cE$, we define  
the function $F_G(u)$ by the conditions
$$
F_G(u) \equiv x(uG) \pmod N \mand 0 \le f_g(u) \le N-1,
$$
where $N$ is the number of $\F_p$-rational points on $\cE$.

We believe that the approach of this paper can also be used
to study fixed points and  cycles of length two and there, 
associated with this map.  However the details can be more 
involved than in the case of modular exponentiation.


\begin{thebibliography}{9999}


\bibitem{BKS} J.~Bourgain, S. V. Konyagin and
I. E. Shparlinski,
`Product sets of rationals, multiplicative
translates of subgroups in residue rings
and fixed points of the discrete logarithm',
{\it Intern.\ Math.\ Research Notices\/},
{\bf 2008} (2008), Article ID rnn090, 1--29
(Corrigenda {\it Intern.\ Math.\ Research Notices\/},
{\bf 2009} (2009), 3146-3147).


\bibitem{CobZah} C. Cobeli and A. Zaharescu, `An exponential congruence with
solutions in primitive roots', {\it Rev. Roumaine Math. Pures Appl.\/},
{\bf 44} (1999),   15--22.

\bibitem{CoSh} D. Coppersmith and I. E. Shparlinski, `On polynomial approximation of the
discrete logarithm and the Diffie--Hellman mapping',
{\it J. Cryptology\/}, {\bf 13} (2000),  339--360.


\bibitem{GoldRos} O. Goldreich and V. Rosen, `On the security
of modular exponentiation with
application to the construction of pseudorandom generators',
{\it J.  Cryptology\/}, {\bf 16} (2003), 71--93.

\bibitem{Hold} J. Holden, `Fixed points and two cycles of the discrete logarithm',
 {\it Lect. Notes in Comp. Sci.\/}, Springer-Verlag, Berlin,   {\bf 2369}
(2002), 405--416.

\bibitem{HoldMor1} J. Holden and P. Moree, `New conjectures and results for small
cycles of the discrete logarithm',
{\it High Primes and Misdemeanours: Lectures in Honour of
the 60th Birthday of Hugh Cowie Williams\/},
Fields Institute Communications, vol.41, Amer. Math. Soc., 2004, 245--254.

\bibitem{HoldMor2} J. Holden and P. Moree, `Some heuristics and
 and results for small cycles of the discrete
logarithm', {\it Math. Comp.\/},  {\bf 75} (2006), 419--449.


\bibitem{MeWi} G. C. Meletiou and A. Winterhof,
`Interpolation of the double discrete logarithm',
{\it Lect. Notes in Comp. Sci.\/}, Springer-Verlag, Berlin,   
{\bf  5130} (2008), 1--10.

\bibitem{PaSu} S. Patel and G. S. Sundaram,
 `An efficient discrete $\log$ pseudo random generator',
{\it  Lect. Notes in Comp. Sci.\/}, Springer-Verlag, Berlin, {\bf 1462} (1999), 35--44.

\bibitem{Silv} J.~H.~Silverman, {\it The arithmetic of elliptic
curves\/},
Springer-Verlag, Berlin, 1995.

\bibitem{Zhang} W. P. Zhang,
`On a problem of Brizolis',
{\it Pure Appl. Math.\/}, {\bf 11} (1995), suppl., 1--3 (in Chinese).
 
\end{thebibliography}
\end{document}